\documentclass[11pt]{amsart}
\usepackage[all]{xy}
\usepackage{amscd}
\usepackage{amsfonts}
\usepackage{amsmath}
\usepackage{amssymb}
\usepackage{amsthm}
\usepackage{color}
\usepackage[center]{caption2}
\usepackage{dsfont}
\usepackage{enumerate}
\usepackage{epic}
\usepackage{float}
\usepackage{geometry}
\usepackage{graphicx}
\usepackage{graphics}
\usepackage{hyperref}
\usepackage{indentfirst}
\usepackage{mathrsfs}
\usepackage{multirow}
\usepackage[numbers,sort&compress]{natbib}
\usepackage{pgf}

\allowdisplaybreaks
\numberwithin{equation}{section}
\theoremstyle{plain}
\newtheorem{prop}{Proposition}[section]
\newtheorem{coro}[prop]{Corollary}

\newtheorem{lemm}[prop]{Lemma}

\newtheorem{theorem}[prop]{Theorem}
\newtheorem*{thrmA}{Theorem A}
\newtheorem*{thrmB}{Theorem B}

\theoremstyle{definition}
\newtheorem{defi}[prop]{Definition}

\newtheorem{exam}[prop]{Example}
\newtheorem{rema}[prop]{Remark}
\renewcommand\aa{a}
\newcommand\AAA{\mathcal{A}}
\newcommand\AD{\mathcal{A}_{\DDD}}

\newcommand\bb{b}
\newcommand\BBB{B}
\newcommand\bl{\vdash}
\newcommand\br{\dashv}

\newcommand\DDD{\mathcal{D}}

\newcommand\ep{\varepsilon}
\newcommand\ff{f}

\newcommand\GK{Gelfand-Kirillov dimension}
\newcommand\gkd[1]{\mathsf{GKdim}(#1)}
\newcommand\gkdim[1]{\mathsf{GKdim}#1}

\newcommand\HS[1]{\leavevmode\null\hspace{#1mm}}
\newcommand\id{\mathsf{Id}}
\newcommand\Id[1]{\id(#1)}

\newcommand\ii{i}
\newcounter{ITEM}
\newcommand\ITEM[1]{\setcounter{ITEM}{#1}\leavevmode\hbox{\rm(\roman{ITEM})}}
\newcommand\jj{j}
\newcommand\kk{k}
\renewcommand\ll{l}
\newcommand\mm{m}

\newcommand\NBd[2]{[#1]_{#2}} 
\newcommand\nn{n}
\newcommand\qq{q}

\newcommand\pdots{\mathrel{\HS{0.2}{\cdot}{\cdot}{\cdot}\HS{0.2}}}
\newcommand\pp{p}

\newcommand\rr{r}

\renewcommand\SS{S}
\newcommand\Span{\mathsf{span}}
\renewcommand\tt{t}

\newcommand\uu{u}

\newcommand\vv{v}
\newcommand\VVV{V}
\newcommand\wdots{, ...\HS{0.2}, }

\newcommand\ww{w}

\newcommand\xx{x}
\newcommand\XXX{X}

\newcommand\Xomega{[\XXX^+]_\omega}
\newcommand\YYY{Y}
\newcommand\yy{y}
\newcommand\zz{z}

\title{No dialgebra has Gelfand-Kirillov dimension strictly between 1 and 2$^{\ddag}$}

\author{Zerui Zhang$^*$}
\address{Z.Z., School of Mathematical Sciences, South China Normal University, Guangzhou 510631, P. R. China}
\email{\small 295841340@qq.com}

\author{Yuqun Chen$^{\sharp}$}
\address{Y.C., School of Mathematical Sciences, South China Normal University, Guangzhou 510631, P. R. China}
\email{yqchen@scnu.edu.cn}

\author{Bing Yu$^{\dagger}$}
\address{Y.B., School of Mathematical Sciences, South China Normal University Guangzhou 510631, P. R. China}
\email{\small 449224830@qq.com}

\thanks{${}^{\ddag}$ Supported by the NNSF of China (11571121), the NSF of Guangdong Province (2017A030313002) and the Science and Technology Program of Guangzhou (201707010137)}

\thanks{${}^*$ Supported by the Innovation Project of Graduate School of South China Normal University}

\thanks{${}^{\sharp}$ Corresponding author}


\begin{document}

\begin{abstract}
The Gelfand-Kirillov dimension measures the asymptotic growth rate of algebras. For every associative dialgebra~$\DDD$,  the quotient~$\AAA_\DDD:=\DDD/\Id\SS$, where~$\Id\SS$ is the ideal of~$\DDD$ generated by the set~$\SS:=\{\xx\bl\yy-\xx\br\yy\mid \xx,\yy\in \DDD\}$,  is called the associative algebra associated to~$\DDD$. Here we show that the \GK\ of~$\DDD$ is bounded above by twice the \GK\ of~$\AAA_\DDD$. Moreover, we prove that no associative dialgebra has Gelfand-Kirillov dimension strictly between 1 and~2.
\end{abstract}
\maketitle
\section{Introduction}
The Gelfand-Kirillov dimension has become one of the important tools in the study of algebras. There are many well known results on \GK s of associative algebras and modules~\cite{book}. For instance, the Gelfand-Kirillov dimension of a finitely generated commutative algebra~$\AAA$ over a field is the classical Krull dimension of~$\AAA$~\cite[Theorem 4.5]{book}.  In particular, the \GK\ of the free commutative algebra generated by~$n$ elements is exactly~$n$, where~$n$ is a positive integer.
 Here we shall investigate the \GK s of associative dialgebras.

Recall that an \emph{associative dialgebra} (dialgebra for short)~$(\DDD, \bl,\br)$ over a field~$k$ is a~$k$-vector space equipped with two bilinear operations
$\bl : \DDD\otimes \DDD\rightarrow \DDD$ and
$\br \ : \DDD\otimes \DDD\rightarrow \DDD$ such that
 $(\DDD,\bl)$ and $(\DDD,\br)$ are associative algebras and the following identities hold:
\begin{equation}
\begin{cases}
\xx \br(\yy \bl \zz)=\xx\br (\yy\br \zz),\\
(\xx\br \yy)\bl \zz=(\xx\bl \yy)\bl \zz, \\
\xx\bl(\yy\br \zz)=(\xx\bl \yy)\br \zz,
\end{cases}
\end{equation}
for all~$\xx, \yy$ and~$\zz$ in~$\DDD$. For instance, let~$(\AAA,\partial)$ be a differential associative algebra satisfying~$\partial^2=0$, and define~$\xx\bl\yy=\partial(\xx) y$ and~$\xx\br\yy=\xx\partial(\yy)$. Then~$(\AAA,\bl,\br)$ becomes a dialgebra~\cite{loday}.

Let~$\DDD$ be a dialgebra and let~$\AAA_\DDD=\DDD/\Id\SS$, where~$\Id\SS$ is the ideal of~$\DDD$ generated by the set~$\SS:=\{\xx\bl\yy-\xx\br\yy\mid \xx,\yy\in \DDD\}$. Then~$\AAA_\DDD$ is an associative algebra, which is called the associative algebra associated to~$\DDD$.  We find that
the \GK~$\gkd\DDD$ of~$\DDD$ and the \GK~$\gkd\AD$ of~$\AD$ have the following relation:
\begin{thrmA}\label{THA}
 Let $\DDD$ be a dialgebra and let~$\AD$ be the associative algebra associated to~$\DDD$.    Then we have
$$
\gkd\AD\leq\gkd\DDD\leq2\gkd\AD.
$$
In particular, the inequality $\gkd\DDD<\infty$ holds if and only if~$\gkd\AD<\infty$ holds.
\end{thrmA}

For every associative algebra~$(\AAA, \cdot)$, we define~$\xx\bl\yy=\xx\br\yy=\xx\cdot\yy$ for all~$\xx,\yy$ in~$\AAA$. Then it is clear that~$(\AAA,\bl,\br)$ becomes a dialgebra. It is well-known that for every real number~$\rr\geq 2$, there exists an associative algebra~$\AAA$ satisfying~$\gkd\AAA=r$,  so for every real number~$\rr\geq 2$, there exists a dialgebra~$\DDD$ satisfying~$\gkd\DDD=r$. Bergman~\cite{Bergman} proved that there is no associative algebra having Gelfand-Kirillov dimension in the open interval~$(1, 2)$.  C. Martinez and E. Zelmanov~\cite{zelmanov} proved an analogous theorem for Jordan algebras.  However, for Lie algebras and for Jordan superalgebras, the situations are quite different: the
Gelfand--Kirillov dimension of a finitely generated Lie algebra~\cite{lie-growth} or of a Jordan superalgebra~\cite{Jordan-super-growth}  can be an arbitrary
number in~$\{0\}\cup[1,\infty]$. This leaves the following problem open: Does there exist a dialgebra of \GK\ in the open interval~$(1,2)$? Based on Bergman's result, we prove the following result:
\begin{thrmB}\label{THB}
 No dialgebra has  Gelfand-Kirillov dimension  strictly between 1 and 2.
\end{thrmB}
Although Theorem B is based on the corresponding result of Bergman, the extension is not obvious because really different techniques are needed here, among which constructing connected linear bases for~$\DDD$ and~$\AD$ simultaneously plays an important role. In general, assuming~$\gkd{\AD}=1$ does not imply~$\gkd\DDD=1$. We shall provide a sufficient condition in Lemma~\ref{special-basis} ensuring~$\gkd\DDD=\gkd\AD$. Moreover, we shall show that, for~$\gkd\AD=1$, the sufficient condition is satisfied.

The paper is organized as follows. In Section~\ref{sec-gk-defi}, we recall some basic properties of a dialgebra and the definition of the \GK.  We then provide a formula, involving the generators of a dialgebra, for calculating the \GK \ (Lemma~\ref{gk-generator}). In Section~\ref{sec-main-result}, we introduce how to construct a shortest-middle-lexicographic (linear) basis of a dialgebra (Definition~\ref{defi-basis}). Such (linear) bases turn out to be very useful in connecting the \GK\ of a dialgebra and that of its associated associative algebra, and they play important roles in the proofs of Theorem A and Theorem B. Finally, to prove Theorem B, we also need to investigate the middle entries (Definition~\ref{defi-mid-ent}) of monomials in the constructed shortest-middle-lexicographic (linear) bases.

\section{The \GK\ of a dialgebra}\label{sec-gk-defi}
Our aim in this section is to recall the notion of the \GK\ of a dialgebra.  We shall first recall some basic notations and properties of a dialgebra. Then we shall introduce several formulas for the \GK\ of a dialgebra.

Recall that for every dialgebra~$\DDD$, for all~$\xx_1\wdots\xx_t$ in~$\DDD$, every parenthesizing of
$$
\xx_1\bl\cdots\bl\xx_m\br\cdots\br\xx_t
$$
 gives the same element in~$\DDD$~\cite{loday}, which we denote by~$[\xx_1...\xx_t]_m$ or~${\xx_1\bl\cdots\bl\xx_m\br\cdots\br\xx_t}$. The following definition of middle submonomials of a monomial  resembles that of subwords of a word.
 \begin{defi}\cite{loday}\label{defi-mid-ent}
For all~$\xx_1\wdots\xx_t$ in~$\DDD$, the element~$\xx_m$ in the monomoial~$[\xx_1...\xx_t]_m$ is called the \emph{middle entry} of~$[\xx_1...\xx_t]_m$, and every monomial~$[\xx_p...\xx_q]_{m-p+1}$ satisfying~$1\leq p\leq m\leq q\leq t$ is called a \emph{middle submonomial} of~$[\xx_1...\xx_t]_m$.
 \end{defi}

 The following formulas for calculating the product of two monomials in a dialgebra will be very useful in the sequel. Roughly speaking, the middle entry of the product of two monomials is the middle entry that the operation~$\bl$ or~$\br$ points to. More precisely, we have the following lemma:
\begin{lemm}\emph{\cite{loday}}\label{product}
For all integers~$\mm,\nn,\pp,\qq$ such that~$1\leq\pp\leq \nn<\mm$ and~$1\leq \qq\leq m-n$,
for all~$\xx_1\wdots\xx_m$ in a dialgebra~$\DDD$,  we have the following formulas:

\ITEM1 $[\xx_1...\xx_n]_p \bl [\xx_{n+1}...\xx_m]_q=[\xx_1...\xx_m]_{n+q}$.

\ITEM2 $[\xx_1...\xx_n]_p \br [\xx_{n+1}...\xx_m]_q=[\xx_1...\xx_m]_{p}$.
\end{lemm}

 Before introducing a general definition of the \GK\ of a dialgebra~$\DDD$ over a field~$\kk$, we have to introduce some more notations.
Let~$\VVV, \VVV_1$ and~$\VVV_2$ be vector subspaces of~$\DDD$. We first define
$$\VVV_1\bl\VVV_2=\Span_k\{\xx\bl\yy\mid \xx\in \VVV_1, \yy\in \VVV_2\} \mbox{ and }\VVV_1\br\VVV_2=\Span_k\{\xx\br\yy\mid \xx\in \VVV_1, \yy\in \VVV_2\}.$$
Then we define~$\VVV^1=\VVV$ and~$\VVV^{n}=\sum_{1\leq i\leq n-1}(\VVV^i\bl\VVV^{n-i}+\VVV^i\br\VVV^{n-i})$ for every integer number~$\nn\geq 2$. Finally, we define
$$
\VVV^{\leq n}:=\VVV^{1}+\VVV^{2}+\pdots+\VVV^{n}.
$$
By Lemma~\ref{product}, we easily obtain
$$
\VVV^{n}=\Span_k\{\NBd{\xx_{1}...\xx_{\nn}}\pp \mid  1\leq p\leq \nn, \mbox{ where }\nn,\pp \in\mathbb{N}\mbox{ and }\xx_{1}\wdots\xx_n\in\VVV  \}
$$
and
$$
\VVV^{\leq n}=\Span_k\{\NBd{\xx_{1}...\xx_{\mm}}\pp \mid 1\leq p\leq \mm\leq\nn, \mbox{ where }\pp, \mm, \nn \in\mathbb{N} \mbox{ and }\xx_{1}\wdots\xx_\mm\in\VVV \}.
$$
Now we are ready to introduce the \GK\ of a dialgebra.
\begin{defi}\label{gkdefi}
Let~$\DDD$ be a dialgebra over a field~$\kk$, the \GK\ of a dialgebra~$\DDD$ is defined to be
$$ \gkd\DDD=\sup\limits_{\VVV}\overline{\lim\limits_{\nn \to \infty}}\log_{\nn}\mathsf{dim}(\VVV^{\leq\nn}),$$
where the supremum is taken over all finite dimensional subspaces~$\VVV$ of~$\DDD$.
\end{defi}
For instance, let~$\DDD$ be the free dialgebra generated by a letter~$\aa$.  Then it is easy to see that~$\{[\aa_1...\aa_n]_p\mid \aa_1=\pdots=\aa_n=\aa, 1\leq \pp\leq \nn$, and~$\pp, \nn \in~\mathbb{N}\}$ is a linear basis of~$\DDD$.   So by direct calculation, we obtain~$\gkd\DDD=2$.

By Definition~\ref{gkdefi}, if~$\DDD'$ is a subalgebra of~$\DDD$ or a homomorphic image of~$\DDD$, then we have~$\gkd{\DDD'}\leq \gkd{\DDD}$. And since we often consider generating sets of a dialgebra, we shall also use the following formula for calculating the \GK\ in the sequel.
\begin{lemm}\label{gk-generator}
Let~$\DDD$ be a dialgebra generated by a set~$\XXX$.
Then we have
\begin{equation}\label{formula}
\gkd\DDD=\sup\limits_{\XXX'\subseteq\XXX}
\overline{{\lim\limits_{n \to \infty}}}\log_{n} \mathsf{dim}((\kk\XXX')^{\leq n}),
\end{equation}
where the supremum is taken over all finite nonempty subsets~$\XXX'$ of~$\XXX$, and~$\kk\XXX'$ is the  vector subspace of~$\DDD$ spanned by~$\XXX'$.
\end{lemm}

\begin{proof}
Let~$\VVV$ be a finite dimensional subspace of~$\DDD$. Then for some finite subset~$\XXX'$ of~$\XXX$ and for some integer number~$\mm$, we have~$\VVV \subseteq (\kk\XXX')^{\leq m}$ and~$\VVV^{\leq\nn}\subseteq ((\kk\XXX')^{\leq\mm})^{\leq\nn}\subseteq(\kk\XXX')^{\leq\mm\nn}$ for every integer number~$\nn\geq 1$.
Consequently, we deduce
\begin{align*}
\overline{{\lim\limits_{\nn \to \infty}}}\log_{\nn} \mathsf{dim}(\VVV^{\leq\nn})
&\leq \overline{{\lim\limits_{\nn \to \infty}}}\log_{\nn} \mathsf{dim}((\kk\XXX')^{\leq\mm\nn})
\overset{\tt=\mm\nn}{=} \overline{{\lim\limits_{\tt \to \infty}}}\frac{\log_{\tt}
\mathsf{dim}((\kk\XXX')^{\leq\tt})}{\log_{\tt}(\tt/\mm)}&
\\
&= \overline{{\lim\limits_{\tt \to \infty}}}\log_{\tt} \mathsf{dim}((\kk\XXX')^{\leq\tt})
\leq \sup\limits_{\XXX'\subseteq\XXX}\overline{{\lim\limits_{\tt \to \infty}}}\log_{\tt} \mathsf{dim}((\kk\XXX')^{\leq\tt})&
\end{align*}
and
$$
\gkd\DDD=\sup\limits_{\VVV}\overline{{\lim\limits_{\nn \to \infty}}}\log_{\nn} \mathsf{dim}(\VVV^{\leq\nn})
\leq\sup\limits_{\XXX'\subseteq\XXX}\overline{{\lim\limits_{\tt \to \infty}}}\log_{\tt} \mathsf{dim}((\kk\XXX')^{\leq\tt}).
$$
On the other hand, for every finite subset~$\XXX'$ of~$\XXX$, the subspace~$\kk\XXX'$ is a finite dimensional subspace of~$\DDD$. By Definition~\ref{gkdefi}, the lemma follows.
\end{proof}
As a consequence, when one applies Equation~\eqref{formula} to calculate the \GK\ of a dialgebra~$\DDD$, it is independent of the choices of the sets of generators of~$\DDD$.
In particular, when~$\DDD$ is finitely generated, Equation~\eqref{formula} becomes easier:
\begin{coro}\label{finite-formula}
Let~$\DDD$ be a dialgebra generated by a finite set~$\XXX$. Then we have $$\gkd\DDD=\overline{{\lim\limits_{\nn\to \infty}}}\log_{\nn}
\mathsf{dim}(  (\kk\XXX)^{\leq\nn}).$$
In particular, for every dialgebra~$\DDD$, we have~$\gkd\DDD=\sup_{\DDD'}\gkd{\DDD'}$,
where the supremum is taken over all finitely generated subalgebras~$\DDD'$ of~$\DDD$.
\end{coro}

 \section{No dialgebra has \GK\ in the open interval~$(1, 2)$}\label{sec-main-result}

Our aim in this section is to prove our main results on \GK s of dialgebras. We shall prove the inequality~$\gkd\AD \leq \gkd\DDD\leq 2\gkd\AD$ and construct dialgebras such that~$\gkd\AD=\gkd\DDD$ or~$\gkd\DDD=2\gkd\AD$. Finally, we shall conclude the paper with Theorem~B.

We first recall the linear basis of a free dialgebra generated by an arbitrary well-ordered set~$\XXX$ constructed by Loday~\cite{loday}. Let~$(\XXX,<)$ be a fixed well-ordered set. Then we use~$\XXX^+$ for the free semigroup generated by~$\XXX$. For every~$\uu=\aa_1...\aa_n$ in~$\XXX^+$, where~$\aa_1\wdots\aa_n$ lie in~$\XXX$, we define the length~$\ell(\uu)$ of~$\uu$ to be~$n$. Finally, let~$\DDD(\XXX)$ be the free dialgebra generated by~$\XXX$. Then the following set
$$
[\XXX^+]_\omega:=\{[\aa_1...\aa_n]_\mm \mid \aa_1\wdots\aa_n\in \XXX,    \mm,\nn\in \mathbb{N}, 1\leq \mm \leq \nn \}
$$
forms a linear basis of~$\DDD(\XXX)$~\cite{loday}. For a nonempty sequence~$\uu=\aa_1...\aa_n$ over~$\XXX$, we shall also use the notation~$\NBd\uu\mm$ for~$\NBd{\aa_1...\aa_n}\mm$, and we call~$\NBd{\aa_1...\aa_n}\mm$ a \emph{disequence} over~$\XXX$. Obviously, two disequences~$\NBd{\aa_1...\aa_n}\mm$ and~$\NBd{\bb_1...\bb_p}\qq$ are the same only if~$\nn=\pp$, $\mm=\qq$ and~$\aa_i=\bb_i$ for every~$\ii\leq \nn$. Finally, for a subset~$\YYY$ of~$[\XXX^+]_{\omega}$, we define~$|\YYY|$ to be cardinality of the set~$\YYY$.

 Now we recall a well-ordering on~$[\XXX^+]_\omega$ introduced in~\cite{shijie}. Here we use a slightly different name for the ordering to make it easy to remember how we compare two monomials.
\begin{defi}\cite{shijie}\label{ordering}
We recall that the \emph{length-middle-lexicographic ordering} $<$ on $[\XXX^+]_\omega$ is defined as follows: For all~$[\uu]_{\pp}=[\aa_{1}...\aa_{\nn}]_{\pp}$ and~$[\vv]_{\qq}=[\bb_1...\bb_m]_{\qq}$  in~$[X^+]_\omega$ such that~$\aa_1\wdots\aa_n$ and~$\bb_1\wdots\bb_m$ lie in~$\XXX$, we define
$$[\uu]_{\pp}<[\vv]_{\qq} \ \mbox{if} \  (\ell(\uu),\pp,\aa_{1}\wdots \aa_{\nn})<(\ell(\vv),\qq,\bb_{1}\wdots \bb_{\mm}) \ \ \mbox{lexicographically}.$$
\end{defi}
For instance, for all~$\aa_1\wdots\aa_4$ in~$\XXX$ such that~$\aa_1< \pdots<\aa_4$, we have~$[\aa_{4}\aa_{3}\aa_{2}]_{1}<[\aa_{1}\aa_{3}\aa_{2}]_{2}$ and~$[\aa_{1}\aa_{2}\aa_{3}\aa_4]_{1}>[\aa_{1}\aa_{2}\aa_{4}]_{2}$.
Note that the above ordering is not compatible with the products in general. For instance,
 we have~$[\aa_1\aa_2]_2>[\aa_2\aa_1]_1$ and~$[\aa_1\aa_2]_2\bl \aa_3<[\aa_2\aa_1]_1\bl\aa_3$.

However, the length-middle-lexicographic ordering still has some good properties as follows:
\begin{lemm}\label{monomial-order}
Let~$\NBd{\uu_1}{\mm_1}$, $\NBd{\uu_2}{\mm_2}$ and~$\NBd{\uu_3}{\mm_3}$ be monomials in~$\Xomega$  with~$\NBd{\uu_1}{\mm_1}<\NBd{\uu_2}{\mm_2}$.  Then we have~$\NBd{\uu_3}{\mm_3}\bl \NBd{\uu_1}{\mm_1}
< \NBd{\uu_3}{\mm_3}\bl \NBd{\uu_2}{\mm_2}$ and~$\NBd{\uu_1}{\mm_1}\br \NBd{\uu_3}{\mm_3}
<  \NBd{\uu_2}{\mm_2}\br \NBd{\uu_3}{\mm_3}$;
if, in addition, $\mm_2=1$, then we also have
~$$\NBd{\uu_3}{\mm_3}\br \NBd{\uu_1}{\mm_1}
< \NBd{\uu_3}{\mm_3}\br \NBd{\uu_2}{\mm_2} \mbox{ and }\NBd{\uu_1}{\mm_1}\bl \NBd{\uu_3}{\mm_3}
<  \NBd{\uu_2}{\mm_2}\bl \NBd{\uu_3}{\mm_3}.$$
\end{lemm}
\begin{proof}
  Note that~$\NBd{\uu_1}{\mm_1}<\NBd{\uu_2}{1}$ implies~$(\ell(u_1),\mm_1)\leq(\ell(\uu_2),1)$. Moreover, if~$\ell(u_1)=\ell(u_2)$, then we obtain~$\mm_1=1$. The remain of the proof is straightforward.
\end{proof}
  Let~$\DDD$ be a dialgebra generated by a well-ordered set~$\XXX$. Then~$[\XXX^+]_{\omega}$ is a linear generating set of~$\DDD$. Moreover, there is no harm to assume that~$\AAA_\DDD$ is also generated by~$\XXX$ and~~$[\XXX^+]_{\omega}$ is also a linear generating set of~$\AAA_\DDD$. To avoid possible confusions, whenever we write down a linear combination~$\ff$ of elements in~$[\XXX^+]_{\omega}$, we shall declare~$\ff$ in~$\DDD$ or~$\AAA_\DDD$. In this way, we are able to use the same notation for an element in~$\DDD$ or in~$\AAA_\DDD$.

We are now ready to introduce a general way of constructing a subset~$\BBB_\DDD$ of~$[\XXX^+]_{\omega}$ for a dialgebra~$\DDD$ with respect to a generating set~$\XXX$, and we shall see in Lemma~\ref{lemma-basis} that~$\BBB_\DDD$ turns out to be a linear basis of~$\DDD$. By convention, we shall assume that the empty set is a linear basis of the trivial vector space~$\{0\}$.
  \begin{defi}\label{defi-basis}
    Let~$\DDD$ be a dialgebra generated by a well-ordered set~$\XXX$. We call
    \begin{equation}\label{way-define-basis}
\parbox{132mm}{$\BBB_\DDD:=\{[\uu]_\pp \in \Xomega \mid $ for all integer number~$\nn$, for all~$[\uu_1]_{\pp_1}\wdots [\uu_n]_{\pp_n}$ in~$\Xomega$ such that~$[\uu_i]_{\pp_i}<[\uu]_\pp$ for every~$i\leq n$, the element~$[\uu]_\pp$ can not be written as a linear combination of the elements $[\uu_1]_{\pp_1}\wdots [\uu_n]_{\pp_n}$ in~$\DDD$ and~$[\uu]_\pp$ is not~$0$ in~$\DDD\}$ }
 \end{equation}
the \emph{shortest-middle-lexicographic basis} of~$\DDD$ with respect to~$\XXX$.
 \end{defi}

By Definition~\ref{defi-basis}, we know that~$\BBB_\DDD$ depends on the generating set~$\XXX$. However, to make the notation and formulas simple,
we shall still use the notation~$\BBB_\DDD$ but not~$\BBB_{\DDD,\XXX}$.  (No confusions arise in the sequel.)

\begin{lemm}\label{lemma-basis}
Let~$\DDD$ be a dialgebra generated by a well-ordered set~$\XXX$. Then the shortest-middle-lexicographic basis~$\BBB_\DDD$ of~$\DDD$ with respect to~$\XXX$ is a linear basis of~$\DDD$. In particular, let~$\AAA=\AAA_\DDD$ be the associative algebra associated to~$\DDD$, and let~$\BBB_\AAA$ be the shortest-middle-lexicographic basis of~$\AAA$ with respect to~$\XXX$. Then~$\BBB_\AAA$ is a linear basis of~$\AAA$.
\end{lemm}
\begin{proof}
 If~$\DDD=\{0\}$, then~$\BBB_\DDD$ is an empty set. Now we assume~$\DDD\neq\{0\}$. To show that~$\BBB_\DDD$ is a linear generating set of~$\DDD$, it suffices to show that every~$[u]_p$ in~$\Xomega$ lies in the subspace~$\Span_k(\BBB_\DDD)$ of~$\DDD$ spanned by~$\BBB_\DDD$.
  We use induction on~$[u]_p$ with respect to the length-middle-lexicographic ordering in Definition~\ref{ordering}. The minimal element of~$\XXX$ that is not~$0$ in~$\DDD$, say~$\aa$, lies in~$\BBB_\DDD$ by the definition of~$\BBB_\DDD$. Suppose~$[\uu]_\pp>\aa$. If~$[\uu]_\pp$ lies in~$\BBB_\DDD$ or~$[\uu]_\pp=0$ in~$\DDD$, then there is nothing to prove. Otherwise, say~$[\uu]_\pp=\sum_{1\leq i\leq n}\alpha_i[\uu_i]_{\pp_i}$ for some elements~$\alpha_1\wdots\alpha_n$ in~$\kk$ and for some monomials~$[\uu_1]_{\pp_1}\wdots[\uu_n]_{\pp_n}$ in~$\Xomega$ such that~$[\uu_i]_{\pp_i}<[\uu]_\pp$ for every~$i\leq n$, then by induction hypothesis, we have~$[\uu]_\pp\in \Span_k(\BBB_\DDD)$.
  Finally, by the construction of~$\BBB_\DDD$, it is clear that the set~$\BBB_\DDD$ is linear independent in~$\DDD$.
\end{proof}

Now we show that, every~$\xx\neq 0$ in~$\DDD$ such that~$\xx=0$ in~$\AAA_\DDD$ is a left zero divisor with respect~$\bl$ and a right zero divisor with respect to~$\br$ in~$\DDD$. To simplify the formulas, we define~$\ep$ to be an empty disequence, and for every~$\yy$ in~$\DDD$, we define~$\ep\bl\yy=\yy\br\ep=\yy$.
\begin{lemm}\label{zero-divisor}
 Let~$\DDD$ be a dialgebra and let~$\AAA_\DDD$ be the associative algebra associated to~$\DDD$.
Then for every~$\xx$ in~$\DDD$ such that~$\xx=0$ in~$\AAA_\DDD$, for every~$\yy$ in~$\DDD$, we have~$\xx\bl\yy=\yy\br\xx=0$ in~$\DDD$.
\end{lemm}
\begin{proof}
  If~$\xx=0$ in~$\AAA_\DDD$, then we have~$
\xx=\sum\alpha_{\jj}(\zz_j\bl(\xx_j\bl\yy_j -\xx_j\br\yy_j))\br\zz_j' \mbox{ in } \DDD$ by Lemma~\ref{product},
where each~$\xx_j$ and each~$\yy_j$ lie in~$\DDD$, each~$\zz_j$ and each~$\zz_j'$ may be elements of~$\DDD$ or be the empty disequence~$\ep$. So for every~$\yy$ in~$\DDD$, by Lemma~\ref{product} again, we obtain
\begin{align*}
  &((\zz_j\bl(\xx_j\bl\yy_j -\xx_j\br\yy_j))\br\zz_j')\bl\yy&
  \\
  =&((\zz_j\bl(\xx_j\bl\yy_j))\br\zz_j')\bl\yy -((\zz_j\bl(\xx_j\br\yy_j))\br\zz_j')\bl\yy&
  \\
  =&((\zz_j\bl(\xx_j\bl\yy_j))\br\zz_j')\bl\yy -((\zz_j\bl(\xx_j\bl\yy_j))\br\zz_j')\bl\yy=0.&
\end{align*}
Similar to the above reasoning, we obtain~$\yy\br\xx=0$.
\end{proof}

 Before going further, we shall observe some more properties on the shortest-middle-lexicographic basis of a dialgebra. In Lemma~\ref{subword-in-a}, we shall use the shortest-middle-lexicographic basis of an associative subalgebra~$\AAA'$ of~$\AAA_\DDD$ generated by~$\XXX'$, where~$\XXX'$ is a nonempty subset of~$\XXX$. By changing every occurrence of~$\DDD$ into~$\AAA'$ and every occurrence of~$\XXX$ into~$\XXX'$ in~\eqref{way-define-basis}, we shall obtain the corresponding set~$\BBB_{\AAA'}$.

\begin{lemm}\label{subword-in-a}Let~$\DDD$ be a dialgebra generated by a well-ordered set~$\XXX$, and let~$\AAA=\AAA_\DDD$ be the associative algebra associated to~$\DDD$. Let~$\DDD'$ be a subalgebra of~$\DDD$ generated by a subset~$\XXX'$ of~$\XXX$, and let~$\AAA'$ be a subalgebra of~$\AAA$ generated by~$\XXX'$.
Then for every monomial~$[\aa_{1}...\aa_{t}]_{\pp}$ in the shortest-middle-lexicographic basis~$\BBB_{\AAA'}$ of~$\AAA'$ with respect to~$\XXX'$, where the letters~$\aa_1\wdots\aa_t$ lie in~$\XXX'$, we have~$\pp=1$. Moreover, if~$[\aa_{1}...\aa_{t}]_{\pp}$ lies in the shortest-middle-lexicographic basis~$\BBB_{\DDD'}$ of~$\DDD'$ with respect to~$\XXX'$ and~$\pp>1$, then~$[\aa_{1}...\aa_{\pp-1}]_{1}$ lies in~$\BBB_{\AAA'}$; if~$[\aa_{1}...\aa_{t}]_{\pp}$ lies in~$\BBB_{\DDD'}$ and~$\pp<\tt$, then~$[\aa_{\pp+1}...\aa_{t}]_{1}$ lies in~$\BBB_{\AAA'}$.
\end{lemm}

\begin{proof}
By the definition of~$\AAA$,  for every integer number~$\pp$ such that~$1\leq \pp\leq \tt$, we have~$[\aa_{1}...\aa_{t}]_{\pp}=[\aa_{1}...\aa_{t}]_{1}$ in~$\AAA$ and thus in~$\AAA'$. So~$[\aa_{1}...\aa_{t}]_{\pp} \in\BBB_{\AAA'}$ forces~$\pp=1$.
For the second claim, we may assume that~$\pp>1$ without loss of generality. Suppose that~$[\aa_{1}...\aa_{\pp-1}]_{1}\notin\BBB_{\AAA'}$.

 If~$[\aa_{1}...\aa_{\pp-1}]_{1}=0$ in~$\AAA'$ and thus in~$\AAA$, then by Lemma~\ref{zero-divisor}, we have
$$[\aa_{1}\aa_{2}\cdots\aa_{\tt}]_{\pp}= [\aa_{1}...\aa_{\pp-1}]_{1} \bl [\aa_\pp...\aa_\tt]_1=0 \mbox{ in } \DDD \mbox{ and hence in } \DDD',$$
which contradicts with the fact that~$[\aa_1...\aa_t]_\pp$ lies in~$\BBB_{\DDD'}$.

If~$[\aa_{1}...\aa_{\pp-1}]_{1}\neq 0$ in~$\AAA'$, then we have~$[\aa_{1}...\aa_{\pp-1}]_{1}=\sum_{1\leq i\leq n}\alpha_{\ii}[\uu_i]_{\pp_i}$ in $\AAA'$ and thus in~$\AAA$ for some monomials~$[\uu_1]_{\pp_1}\wdots[\uu_n]_{\pp_n}$ in~$[{\XXX'}^+]_\omega$ satisfying~$[\uu_i]_{\pp_i}<[\aa_{1}...\aa_{\pp-1}]_{1}$ for every integer~$i\leq n$ and for some elements~$\alpha_1\wdots\alpha_n$ in the field~$\kk$. In particular,  by Lemma~\ref{zero-divisor} again, we obtain
$$
[\aa_{1}...\aa_{\tt}]_{p}
-\sum_{1\leq i\leq n}\alpha_{\ii}[\uu_{\ii}]_{\pp_i}\bl[\aa_{\pp}\aa_{\pp+1}...\aa_{\tt}]_1
=([\aa_{1}...\aa_{\pp-1}]_{1}-\sum_{1\leq i\leq n}\alpha_{\ii}[\uu_{\ii}]_{\pp_i})\bl[\aa_{\pp}\aa_{\pp+1}...\aa_{\tt}]_1=0
$$
 in~$\DDD$, and thus in~$\DDD'$. Finally, by Lemma~\ref{monomial-order}, we obtain~$[\uu_{\ii}]_{\pp_i}\bl[\aa_{\pp}\aa_{\pp+1}...\aa_{\tt}]_{1}<[\aa_1...\aa_t]_\pp$,
which contradicts with the construction of~$\BBB_{\DDD'}$.
\end{proof}

Note that the length-middle-lexicographic ordering on~$[\XXX^+]_\omega$ plays an important role in Lemma~\ref{subword-in-a}.  Using another way of constructing linear bases for~$\DDD$ and~$\AD$ simultaneously may not have the claimed properties in the lemma.

\begin{rema}
  With the notations of Lemma~\ref{subword-in-a}, if~$\XXX\neq\XXX'$, then~$\AAA'$ is not necessary the associative algebra associated to~$\DDD'$ in general. For instance, let~$\XXX=\{\aa,\bb\}$, and let~$\DDD(\aa,\bb)$ be the free dialgebra generated by~$\XXX$. Then we define~$\DDD$ to be the quotient of~$\DDD(\aa,\bb)$ by the ideal generated by the set~$\{\bb=\aa\bl\aa-\aa\br\aa\}$. It is clear that~$\DDD$ is isomorphic to the free dialgebra generated by~$\{\aa\}$ and~$\AAA$ is isomorphic to the free associative algebra generated by~$\{\aa\}$. Finally, for~$\XXX'=\{\bb\}$, we deduce that, ~$\DDD'$ and the associative algebra associated to~$\DDD'$ are the linear space spanned by~$\bb$ while~$\AAA'$ is the 0 space.
\end{rema}

Let~$\DDD$ be a dialgebra generated by a finite well-ordered set~$\XXX$ and let~$\BBB_\DDD$ be the shortest-middle-lexicographic basis of~$\DDD$ with respect to~$\XXX$. For every positive integer~$n$, we define
$$
\BBB_\DDD^{n}=\{[\uu]_\pp \in \BBB_\DDD \mid \ell(u)=n\}, \ \BBB_\DDD^{\leq n}=\{[\uu]_\pp \in \BBB_\DDD \mid \ell(u)\leq n\}.
$$
 Then we have~$\mathsf{dim}((\kk\XXX)^{\leq \nn})=|\BBB_\DDD^{\leq n}|$ for every~$\nn\geq 1$, where~$|\BBB_\DDD^{\leq n}|$ is the cardinality of the set~$\BBB_\DDD^{\leq n}$. In particular, by Corollary~\ref{finite-formula}, when~$\DDD$ is generated by a finite well-ordered set~$\XXX$, we deduce the following formula for calculating the \GK\ of~$\DDD$:
\begin{equation}\label{gk-by-basis}
   \gkd\DDD=\overline{{\lim\limits_{\nn\to \infty}}}\log_{\nn}|\BBB_\DDD^{\leq\nn}|.
\end{equation}

Now we are ready to prove Theorem~A of the introduction.
\begin{theorem}\label{inequality}
 Let $\DDD$ be a dialgebra and let~$\AD$ be the associative algebra associated to~$\DDD$.    Then we have
\begin{equation}\label{boundary}
\gkd\AD\leq\gkd\DDD\leq2\gkd\AD.
\end{equation}
In particular, the inequality $\gkd\DDD<\infty$ holds if and only if~$\gkd\AD<\infty$ holds.
\end{theorem}

\begin{proof}
We use the notations of Lemma~\ref{subword-in-a}. By Definition~\ref{gkdefi},  it is easy to see that, for every homomorphic image~$\DDD_1$ of~$\DDD$,  we have~$\gkd{\DDD_1}\leq \gkd\DDD$. In particular, we obtain~$\gkd{\AAA_\DDD}\leq\gkd\DDD$.

Now we turn to the inequality~$\gkd\DDD\leq2\gkd\AD$. For finitely generated algebras~$\DDD'$ and~$\AAA'$ as in Lemma~\ref{subword-in-a}, we define~$|\BBB_{\AAA'}^0|=|\BBB_{\DDD'}^0|=1$ for convenience. Then by Lemma~\ref{subword-in-a}, for every integer number~$\tt\geq 1$, we obtain
~$
|\BBB_{\DDD'}^{\tt}|\leq\sum\limits_{1\leq\pp\leq\tt}|\BBB_{\AAA'}^{\pp-1}||\XXX'||\BBB_{\AAA'}^{\tt-\pp}|
$
and
\begin{align*}
&|\BBB_{\DDD'}^{\leq\nn}|=\sum\limits_{1\leq\tt\leq\nn}|\BBB_{\DDD'}^{\tt}|
\leq  \sum_{1\leq\tt\leq\nn}\sum_{1\leq\pp\leq\tt}|
\BBB_{\AAA'}^{\pp-1}||\XXX'||\BBB_{\AAA'}^{\tt-\pp}|
=|\XXX'|\sum_{1\leq\tt\leq\nn}\sum_{\substack{ \ii\geq 0;\jj\geq 0;\\ \ii+\jj=\tt-1}}|\BBB_{\AAA'}^{\ii}||\BBB_{\AAA'}^{\jj}|
&\\
=&|\XXX'|\sum_{0\leq\ii\leq\nn-1}(|\BBB_{\AAA'}^{i}|\sum_{0\leq j\leq n-1-i}
|\BBB_{\AAA'}^{j}|)
\leq |\XXX'|\sum_{0\leq\ii\leq\nn-1}(|\BBB_{\AAA'}^{i}|(|\BBB_{\AAA'}^{\leq n-1}|+|\BBB_{\AAA'}^{0}|) )
&\\
\leq& |\XXX'|(|\BBB_{\AAA'}^{\leq n-1}|+|\BBB_{\AAA'}^{0}|)\sum_{0\leq\ii\leq\nn-1}|\BBB_{\AAA'}^{i}|
\leq |\XXX'|(|\BBB_{\AAA'}^{\leq n-1}|+|\BBB_{\AAA'}^{0}|)^2\leq |\XXX'|(|\BBB_{\AAA'}^{\leq n}|+1)^{2}.&
\end{align*}
Consequently, we obtain
$$
\gkd{\DDD'}=\overline{{\lim\limits_{\nn\to \infty}}}\log_{\nn}|\BBB_{\DDD'}^{\leq\nn}|
\leq\overline{{\lim\limits_{\nn\to \infty}}}\log_{\nn}(|\XXX'|(|\BBB_{\AAA'}^{\leq n}|+1)^{2})
=2\gkd{\AAA'}\leq 2\gkd\AD.
$$
By Corollary~\ref{finite-formula}, we deduce~$\gkd\DDD=\sup_{\DDD'}\gkd{\DDD'}\leq2\gkd\AD$,
where the supremum is taken over all finitely generated subalgebras~$\DDD'$ of~$\DDD$.
\end{proof}

\begin{coro}
  Let~$\DDD$ be a dialgebra. Then~$\gkd\DDD<1$ implies that~$\gkd\DDD=0$.
\end{coro}
\begin{proof}
  Let~$\AAA_\DDD$ be the associative algebra associated to~$\DDD$. Then~$\gkd\DDD<1$ implies that~$\gkd{\AAA_\DDD}=0$. So we obtain~$0=\gkd{\AAA_\DDD}\leq \gkd\DDD\leq 2\gkd{\AAA_\DDD}=0$.
\end{proof}
Now we construct examples to show that the boundaries in~Theorem~\ref{inequality} are the best that one can expect.
We shall describe a sufficient condition ensuring~$\gkd\DDD=\gkd{\AAA_\DDD}$ in Lemma~\ref{special-basis}.
\begin{exam}
Let~$\DDD$ be the free dialgebra generated by a letter~$\aa$, and let~$\AAA_\DDD$ be the associative algebra associated to~$\DDD$. Then we have~$\gkd\DDD=2\gkd{\AAA_\DDD}=2$.
\end{exam}
\begin{proof}
 Note that~$\BBB_\DDD=\{[\aa_1...\aa_n]_p\mid \aa_1=\pdots=\aa_n=\aa, 1\leq \pp\leq \nn$, and~$\pp, \nn \in~\mathbb{N}\}$, and~$\AAA_\DDD$ is the free commutative algebra generated by~$\aa$. We have~$\gkd\DDD=2\gkd{\AAA_\DDD}=2$ immediately by Equation~\eqref{gk-by-basis}.
\end{proof}

\begin{lemm}\label{special-basis}
Let~$\DDD$ be a dialgebra generated by a finite well-ordered set~$\XXX$, and let~$\AAA_\DDD$ be the associative algebra associated to~$\DDD$. If there exists a positive integer~$\mm$ such that~$\DDD$ has a linear generating set~$\BBB=\{[\aa_{1}...\aa_{\nn}]_{\pp}\mid\aa_{1}\wdots\aa_\nn\in\XXX, 1\leq \pp\leq\mm$ or~$0\leq \nn-\pp\leq\mm-1$, where~$\pp,\nn$ lies in~$\mathbb{N}$ and~$\pp\leq \nn\}$. Then we obtain~$\gkd\DDD=\gkd{\AAA_\DDD}$.
\end{lemm}
\begin{proof}
For simplicity, we denote~$\AD$ by~$\AAA$. Since~$\gkd{\AAA}\leq \gkd\DDD$, it suffices to show~$\gkd\DDD\leq \gkd{\AAA}$. Let~$\BBB_\AAA$ be the shortest-middle-lexicographic basis of~$\AAA$ with respect to~$\XXX$ and let~$\BBB_1$ be the subset of~$\BBB$ defined as follows
$$\BBB_1=\{[\aa_{1}...\aa_{\nn}]_{\pp}\in\BBB\mid [\aa_1...\aa_{\pp-1}]_1\in\BBB_\AAA \mbox{ if }\pp>1, \mbox{ and }[\aa_{\pp+1}...\aa_{\nn}]_1\in\BBB_\AAA \mbox{ if }\pp<\nn\}.$$ We claim that~$\BBB_1$ is a linear generating set of~$\DDD$. We use induction with respect to the length-middle-lexicographic ordering on~$\Xomega$ to show that every element~$[\aa_{1}...\aa_{\nn}]_{\pp}$ of~$\BBB$ can be written as a linear combination of elements of~$\BBB_1$. It is clear that~$\XXX$ is a subset of~$\BBB_1$. Without loss of generality, we assume that~$\pp>1$ and~$[\aa_1...\aa_{\pp-1}]_1\notin \BBB_\AAA$. The following proof is similar to Lemma~\ref{subword-in-a}. If~$[\aa_{1}...\aa_{\pp-1}]_{1}=0$ in~$\AAA$, then we have~$[\aa_{1}...\aa_{\nn}]_{\pp}=0$ in~$\DDD$. If~$[\aa_{1}...\aa_{\pp-1}]_{1}\neq 0$ in~$\AAA$, then we may
assume that~$[\aa_{1}...\aa_{\pp-1}]_{1}=\sum_{1\leq i\leq n}\alpha_{\ii}[\uu_i]_{\pp_i}$ in~$\AAA$ for some monomials~$[\uu_1]_{\pp_1}\wdots[\uu_n]_{\pp_n}$ in~$[{\XXX}^+]_\omega$ satisfying~$[\uu_i]_{\pp_i}<[\aa_{1}...\aa_{\pp-1}]_{1}$ for every integer~$i$ such that~$1\leq i\leq n$ and for some elements~$\alpha_1\wdots\alpha_n$ in the field~$\kk$. By Lemma~\ref{monomial-order}, we obtain~$[\uu_{\ii}]_{\pp_i}\bl[\aa_{\pp}\aa_{\pp+1}\cdots\aa_{\nn}]_{1}
<[\aa_1...\aa_n]_\pp$. Finally, we have
$$
[\aa_{1}...\aa_{\nn}]_{\pp}=[\aa_{1}...\aa_{\pp-1}]_{1}\bl[\aa_{\pp}\aa_{\pp+1}..\aa_{\nn}]_{1}
=\sum_{1\leq i\leq n}\alpha_{\ii}[\uu_{\ii}]_{\pp_i}\bl[\aa_{\pp}\aa_{\pp+1}\cdots\aa_{\nn}]_{1}.
$$
 Since~$\ell(\uu_i)\leq \pp-1<m$, each monomial~$[\uu_{\ii}]_{\pp_i}\bl[\aa_{\pp}\aa_{\pp+1}\cdots\aa_{\nn}]_{1}$ lies in~$\BBB$. By induction hypothesis, each monomial~$[\uu_{\ii}]_{\pp_\ii}\bl[\aa_{\pp}\aa_{\pp+1}\cdots\aa_{\nn}]_{1}$, and thus~$[\aa_1...\aa_n]_\pp$, can be written as a linear combination of elements of~$\BBB_1$.

 For every~$\tt\geq 1$, define~$\BBB_1^{\tt}=\{[\aa_1...\aa_t]_p \in \BBB_1\mid   \aa_1\wdots\aa_t \in \XXX, \pp\in~\mathbb{N}\}$, $\BBB_1^{\leq t}=\cup_{1\leq i\leq t}\BBB_1^i$ and define~$|\BBB_1^0|=1$. Then for every~$\tt\geq 1$, we have
$$
|\BBB_{1}^{\tt}|\leq\sum_{\substack{1\leq\pp\leq\mm;\\\pp\leq\tt}}|\BBB_{\AAA}^{\pp-1}||\XXX||\BBB_{\AAA}^{\tt-\pp}|
+\sum_{\substack{t-m+1\leq\pp\leq\tt;\\\pp\geq 1}}|\BBB_{\AAA}^{\pp-1}||\XXX||\BBB_{\AAA}^{\tt-\pp}|
=2\sum_{\substack{1\leq\pp\leq\mm;\\\pp\leq\tt}}|\BBB_{\AAA}^{\pp-1}||\XXX||\BBB_{\AAA}^{\tt-\pp}|,
 $$
and
\begin{align*}
&|\BBB_{1}^{\leq\nn}|=\sum\limits_{1\leq\tt\leq\nn}|\BBB_{1}^{\tt}|
\leq  2\sum_{1\leq\tt\leq\nn}\sum_{\substack{1\leq\pp\leq\mm;\\\pp\leq\tt }}|\BBB_{\AAA}^{\pp-1}||\XXX||\BBB_{\AAA}^{\tt-\pp}|
=2|\XXX|\sum_{1\leq\tt\leq\nn}\sum_{\substack{0\leq\ii\leq m-1;\jj\geq 0;\\ \ii+\jj=\tt-1}}|\BBB_{\AAA}^{\ii}||\BBB_{\AAA}^{\jj}|&
\\
=&2|\XXX|\sum_{\substack{0\leq\ii\leq\mm-1;\\\ii\leq \nn-1}}(|\BBB_{\AAA}^{i}|\sum_{0\leq j\leq n-1-i}
|\BBB_{\AAA}^{j}|)
\leq 2|\XXX|\sum_{0\leq\ii\leq\mm-1}(|\BBB_{\AAA}^{i}|(|\BBB_{\AAA}^{\leq n-1}|+|\BBB_{\AAA}^{0}|) )&
\\
\leq &2|\XXX|(|\BBB_{\AAA}^{\leq n-1}|+|\BBB_{\AAA}^{0}|)\sum_{0\leq\ii\leq\mm-1}|\BBB_{\AAA}^{i}|
\leq 2|\XXX|(|\BBB_{\AAA}^{\leq n-1}|+|\BBB_{\AAA}^{0}|)(|\BBB_{\AAA}^{\leq m-1}|+|\BBB_{\AAA}^{0}|).&
\end{align*}
Finally, we obtain
$$
\overline{{\lim\limits_{\nn\to \infty}}}\log_{\nn}  |\BBB_1^{\leq\nn}|
\leq\overline{{\lim\limits_{\nn\to \infty}}}\log_{\nn}(2|\XXX|(|\BBB_{\AAA}^{\leq n-1}|+|\BBB_{\AAA}^{0}|)(|\BBB_{\AAA}^{\leq m-1}|+|\BBB_{\AAA}^{0}|))
=\overline{{\lim\limits_{\nn\to \infty}}}\log_{\nn}|\BBB_{\AAA}^{\leq n-1}|.
$$
By Corollary~\ref{finite-formula}, we obtain~$\gkd\DDD\leq \gkd{\AAA}$.
\end{proof}

There are many dialgebras satisfying the assumptions of Lemma~\ref{special-basis}. In particular, certain identities ensure the validity of the assumptions of Lemma~\ref{special-basis}.

\begin{coro}\label{yb4}
Let~$\DDD$ be a dialgebra generated by a finite set~$\XXX$, and let~$\AAA_\DDD$ be the associative algebra associated to~$\DDD$. If~$\DDD$ satisfies one of the  identities:
\ITEM1 $\xx\bl\yy=\yy\bl\xx$ for all~$\xx, \yy\in\DDD$;
\ITEM2 $\xx\br\yy=\yy\br\xx$ for all~$\xx, \yy\in\DDD$;
\ITEM3 $\xx\bl\yy=\yy\br\xx$ for all~$\xx, \yy\in\DDD$,
then we have~$\gkd\DDD=\gkd{\AAA_\DDD}\leq|\XXX|$, in particular, $\gkd\DDD$ is a nonnegative integer.
\end{coro}
\begin{proof}
 We shall show that, if one of the identities~\ITEM1-\ITEM3 holds, then the assumptions of Lemma~\ref{special-basis} are valid. Without loss of generality, assume that identity~\ITEM1 holds in~$\DDD$. Then for every~$[\aa_{1}...\aa_{\tt}]_p$  in~$\DDD$ such that~$\pp>1$,
 we have
 $$
 [\aa_{1}...\aa_{\tt}]_p=[\aa_{1}...\aa_{\pp-1}]_{\pp-1} \bl [\aa_\pp...\aa_\tt]_1
= [\aa_\pp...\aa_\tt]_1 \bl [\aa_{1}...\aa_{\pp-1}]_{\pp-1} =[\aa_\pp...\aa_\tt\aa_{1}...\aa_{\pp-1}]_\tt.
$$ In particular, the set~$\BBB_1=\{[\aa_{1}...\aa_{\nn}]_{\pp}\mid\aa_{1}\wdots\aa_\nn\in\XXX$,  $\pp=1$ or~$\pp=n$, $\nn\in \mathbb{N}\}$ is a linear generating set of~$\DDD$. By Lemma~\ref{special-basis}, we obtain~$\gkd\DDD=\gkd{\AAA_\DDD}$.
Moreover, since~$\AAA_\DDD$ is a commutative algebra generated by~$\XXX$, the Gelfand-Kirillov dimension~$\gkd{\AAA_\DDD}$ is a nonnegative integer satisfying~$\gkd{\AAA_\DDD}\leq |\XXX|$.
\end{proof}

Recall that a dialgebra~$(\DDD,\vdash,\dashv)$ is commutative~\cite{Zhuchok} if both~$\vdash$ and $\dashv$ are commutative, that is, for all~$\xx$ and~$\yy$ in~$\DDD$, we have~$\xx\bl\yy=\yy\bl\xx$ and~$\xx\br\yy=\yy\br\xx$.

\begin{coro}
Let~$\DDD$ be a finitely generated commutative dialgebra. Then~$\gkd\DDD$ is a nonnegative integer number. In particular, if~$\DDD$ is a free commutative dialgebra generated by~$\nn$ letters, then we have~$\gkd\DDD=\nn$.
\end{coro}

Now we shall conclude the paper with our main result (Theorem B of the introduction), which is based on Bergman's corresponding result for associative algebras.
\begin{theorem}
No dialgebra has \GK\ strictly between 1 and 2.
\end{theorem}
\begin{proof}
We first assume that~$\DDD$ is a dialgebra generated by a finite well-ordered set~$\XXX$ satisfying the inequality~$\gkd\DDD<2$, and we shall show that~$\gkd\DDD=\gkd{\AAA_\DDD}$.
Let~$\BBB_\DDD$ be the shortest-middle-lexicographic basis of~$\DDD$ with respect to~$\XXX$. Then for some positive integer~$\mm$, we have~$|\BBB_\DDD^{\mm}|<\mm$,  for otherwise we would have~$|\BBB_\DDD^{\leq\nn}|\geq1+2+\cdots+\nn$ for all~$\nn$ and
thus~$\gkd\DDD\geq2$. Define
$$\BBB=\{[\aa_{1}...\aa_{\tt}]_{\pp}\mid\aa_{1}\wdots\aa_\tt\in\XXX,\pp,\tt\in\mathbb{N}, \pp\leq \tt, 1\leq \pp\leq\mm \mbox{ or } 0\leq \tt-\pp\leq\mm-1\}.$$ We shall show that~$\BBB_\DDD$ is a subset of~$\BBB$ and thus by Lemma~\ref{special-basis} we deduce~$\gkd\DDD=\gkd{\AAA_\DDD}$.

For every~$[\aa_{1}...\aa_{\tt}]_{\pp}$ in~$\BBB_\DDD$, if~$\tt\leq 2m$, then it is clear that~$[\aa_{1}...\aa_{\tt}]_{\pp}$ lies in~$\BBB$. Assume now that~$\tt>2m$ and~$\mm<\pp< t-m+1$, then we deduce~$\pp+\mm-1<\tt$ and~$\pp-(\mm-1)>1$. Therefore, for every integer~$\jj$ such that~$0\leq j\leq \mm-1$, the monomial
$$
[\uu_\jj]_{\mm-\jj}:
=[\aa_{p+\jj-(\mm-1)}\aa_{p+\jj-(\mm-1)+1}...\aa_{\pp+\jj}]_{\mm-\jj}
$$
is a middle submonomial of~$[\aa_{1}...\aa_{\tt}]_{\pp}$, and we have~$[\uu_{\ii}]_{\mm-\ii}<[\uu_{\jj}]_{\mm-\jj}$ if~$\ii>\jj$. Since~${|\BBB_\DDD^{\mm}|<\mm}$,   the set~$\{[\uu_{\jj}]_{\mm-\jj}\mid 0\leq \jj\leq m-1\} \cup\BBB_\DDD^{\leq\mm-1}$ is linear dependent in~$\DDD$.
In other words, for some integer~$\nn$ such that~$0\leq \nn\leq \mm-1$, we have
$$
[\uu_{\nn}]_{\mm-\nn}
=-\sum_{n+1\leq\jj\leq \mm-1}\alpha_{\jj}[\uu_{\jj}]_{\mm-\jj}
-\sum_{1\leq\ii\leq\ll}\beta_{\ii}[\ww_{\ii}]_{\qq_{\ii}},$$
for some elements~$\alpha_{n+1}\wdots\alpha_{\mm-1},\beta_{1}\wdots\beta_{\ll}$ in~$\kk$ and for some
monomials~$[\ww_{1}]_{\qq_{1}}\wdots[\ww_{\ll}]_{\qq_{\ll}}$
in~$[\XXX^+]_{\omega}$ such that~$[\ww_{\ii}]_{\qq_{\ii}}<[\uu_{\nn}]_{\mm-\nn}$ for every~$\ii\leq \ll$.
Finally, we obtain
\begin{multline}\label{eqrewrite}
[\aa_{1}...\aa_{\tt}]_{\pp}
=-\sum_{n+1\leq\jj\leq \mm-1}\alpha_{\jj}([\aa_1...\aa_{\pp+\nn-(\mm-1)-1}]_1\bl
[\uu_{\jj}]_{\mm-\jj})
\br[\aa_{\pp+\nn+1}...\aa_{\tt}]_{1}
\\
-\sum_{1\leq\ii\leq\ll}\beta_{\ii}([\aa_1...\aa_{\pp+\nn-(\mm-1)-1}]_1\bl
[\ww_{\ii}]_{\qq_{\ii}})\br[\aa_{\pp+\nn+1}...\aa_{\tt}]_{1}.
\end{multline}
Since each~$[v]_q$ on the right hand side of Equation~\eqref{eqrewrite} satisfies~${[v]_q<[\aa_{1}...\aa_{\tt}]_{\pp}}$, we obtain a contradiction with the definition of~$\BBB_\DDD$.
So we have~$\BBB_\DDD\subseteq \BBB$ and~$\gkd\DDD=\gkd{\AAA_\DDD}$. Finally, the inequality~$1\leq \gkd\DDD<2$ forces~$\gkd\DDD=\gkd{\AAA_\DDD}=1$.

Assume now that~$\DDD$ is not finitely generated. It is clear that for every finitely generated subalgebra~$\DDD'$ of~$\DDD$, we have~$\gkdim(\DDD')\leq \gkd\DDD<2$. By the above reasoning, we deduce~$\gkdim(\DDD')=0$ or~$1$. Moreover, by Corollary~\ref{finite-formula}, we obtain~$\gkd\DDD\leq 1$.  Finally, the inequality~$1\leq \gkd\DDD<2$ forces~$\gkd\DDD=1$.
\end{proof}

For summary, the set of all the possible Gelfand--Kirillov dimensions of dialgebras is~$\{0,1\}\cup[2, \infty)$.

What remains open in view of the previous results is the question of whether the boundaries of inequation~\eqref{boundary} are the only possible values of the \GK\  of a dialgebra, that is, can one show that either~$\gkd\DDD=\gkd\AD$ or~$\gkd\DDD=2\gkd\AD$ holds? We have a feeling that this is true, but so far, we have no proof in this direction.

\end{document}